\documentclass[11pt]{amsart}
\usepackage{}
\usepackage{amssymb,amsmath,amsfonts,amsthm,hyperref,here,enumerate,xfrac,mathtools,graphicx,pstool,xcolor}
\numberwithin{equation}{section}

\newcommand{\I}{\mathrm{I}}

\newcommand{\R}{\mathbb R}
\newcommand{\N}{\mathbb N}

\newtheoremstyle{plain}
  {10pt}
  {10pt}
  {\it}
  {0pt}
  {\bf}
  {}
  {\newline}
  {}
\newtheoremstyle{definition}
  {10pt}
  {10pt}
  {}
  {0pt}
  {\bf}
  {}
  {\newline}
  {}
\theoremstyle{plain}
\definecolor{MyDarkBlue}{rgb}{0,0.29,0.7}
\hypersetup{
    colorlinks=true,       
    linkcolor=MyDarkBlue,          
    linkbordercolor=MyDarkBlue,
    citecolor=MyDarkBlue,
    citebordercolor=MyDarkBlue,        
    filecolor=MyDarkBlue,      
    urlcolor=MyDarkBlue           
}

\theoremstyle{plain}

\newtheorem{theorem}{Theorem}[section]
\newtheorem{coro}[theorem]{Corollary}
\newtheorem{lemma}[theorem]{Lemma}
\newtheorem{prop}[theorem]{Proposition}

\theoremstyle{definition}

\newtheorem{definition}[theorem]{Definition}

\newtheorem{remark}[theorem]{Remark}

\begin{document}

\title{Equidimensional Isometric Extensions}
\author{Micha Wasem}
\address{School of Engineering and Architecture of Fribourg, HES-SO University of Applied Sciences and Arts Western Switzerland, P\'erolles 80, 1700 Fribourg, Switzerland}
\email{micha.wasem@hefr.ch}
\date{\today}
\begin{abstract}
Let $\Sigma$ be a hypersurface in an $n$-dimensional Riemannian manifold $M$, $n\geqslant 2$. We study the isometric extension problem for isometric immersions $f:\Sigma\to\mathbb R^n$, where $\mathbb R^n$ is equipped with the Euclidean standard metric. We prove a general curvature obstruction to the existence of merely differentiable extensions and an obstruction to the existence of Lipschitz extensions of $f$ using a length comparison argument. Using a weak form of convex integration, we then construct one-sided isometric Lipschitz extensions of which we compute the Hausdorff dimension of the singular set and obtain an accompanying density result. As an application we obtain the existence of infinitely many Lipschitz isometries collapsing the standard two-sphere to the closed standard unit $2$-disk mapping a great-circle to the boundary of the disk.\end{abstract}
\maketitle
\section{Introduction}
In this article we analyze the problem of extending a given smooth isometric immersion $f:\Sigma\to \R^q$ of a hypersurface $\Sigma\subset (M^n,g)$ into $\R^q$ equipped with the standard Euclidean metric $g_0=\langle\cdot\,,\cdot\rangle$ in the equidimensional case $q=n$, i.e.\ we will look for a map $v:U\to\R^n$ satisfying
\begin{equation}\label{dissproblem}
\begin{aligned}v^*g_0 & = g\\
v|_\Sigma & = f,
\end{aligned}
\end{equation}
where $U\subset M$ is a neighborhood of a point in $\Sigma$. For high codimension, Jacobowitz established in \cite{jacobowitz} conditions on $\Sigma$ and $f$ such that problem \eqref{dissproblem} admits an analytic $(q\geqslant n(n+1)/2)$ respectively smooth $(q\geqslant n(n+3)/2)$ solution. He also provided a curvature obstruction to the existence of $C^2$-solutions to \eqref{dissproblem}. In \cite{wasem} it is proven that this obstruction is also an obstruction to $C^1$-solutions and in the present paper we will prove along the same lines that it is also an obstruction to Lipschitz solutions (Proposition \ref{lipschitzobstruction}). However, restricting the neighborhood $U$ to one side of $\Sigma$ only, one can hope to construct one-sided solutions of low regularity, where curvature does not exist. Using a Nash-Kuiper iteration, it is possible to construct \emph{one-sided isometric $C^1$-extensions} that satisfy a $C^0$-dense parametric $h$-principle in codimension greater than one i.e.\ if $q\geqslant n+1$ (see \cite{wasem}). Very recently, this result has been significantly improved by Cao and Sz\'ekelyhidi to extensions of class $C^{1,\alpha}$ for $\alpha < 1/(n(n+1)+1)$ in \cite{caolaszlo}. It is worth noticing that the improvement from $C^1$ to $C^{1,\alpha}$ needs a substantially more sophisticated iteration than the one used in \cite{wasem} and that the optimal regularity for this type of result is still open.
\\
\\
This type of aforementioned flexibility is not expected in the equidimensional case since the classical Liouville theorem (see for example \cite[pp.\ 30-31]{ciarlet}) implies that any two images of isometric $C^1$-immersions into Euclidean space are congruent, so even if $g$ is flat, there will not be $C^1$-solutions to \eqref{dissproblem} in general. In this case, it seems natural to relax the regularity and consider piecewise $C^1$-maps (think of folding a piece of paper). Indeed, Dacorogna, Marcellini and Paolini (\cite{dacorogna1}, \cite{dacorogna2}) constructed piecewise isometric $C^1$-immersions from a flat square into $\R^2$ mapping the boundary to a single point.\\
\\
It is remarkable that for general metrics, there is a curvature obstruction even at low regularity in the equidimensional case. We will show that there are no differentiable isometric extensions if $g$ is not flat (Proposition \ref{c1curvature}). Hence we will further relax the regularity and focus on Lipschitz maps instead. Problem \eqref{dissproblem} is replaced thus by the problem of finding a Lipschitz map $v:U\to\R^n$ that satisfies
\begin{equation}\label{dissproblemrelaxed}
\begin{aligned}v^*g_0 & = g\quad\text{a.e.}\\
v|_\Sigma & = f,
\end{aligned}
\end{equation}
where a.e.\ refers to the measure on $M$ induced by $g$. This setting provides enough flexibility to circumvent the curvature obstruction and the one given by Liouville's theorem. In \cite{sverak1} and \cite{sverak2}, M\"uller and \v{S}ver\'ak constructed solutions to the following related Dirichlet problem, where $\Omega$ is a bounded domain in $\R^n$:
$$
\begin{aligned}v^*g_0 & = \operatorname{id}\text{ $\mathcal L^n$-a.e.\ in }\Omega\\
v & = \varphi \text{ on }\partial\Omega,
\end{aligned}
$$
where $\mathcal L^n$ denotes the $n$-dimensional Lebesgue measure. They require the boundary datum $\varphi$ to be a \emph{short map} (see below for a definition), which is not the case for our boundary datum $f$. Maps satisfying \eqref{dissproblemrelaxed} may collapse entire submanifolds to single points, hence this definition does not reflect a truly geometric notion of isometry.
A more natural notion of isometry (which is equivalent if $v\in C^1$) arises if one requires an isometry to be a map that preserves the length of every rectifiable curve (see \cite{spadaro}) and an even stronger notion of isometry is discussed in \cite{petrunin2}. In \cite{spadaro}, a flexible construction of equidimensional isometric maps is given in the more natural setting of length-preserving isometries. The results obtained in the present work do not extend to this stronger setting and the techniques used in \cite{spadaro} lead to stronger results than the ones obtained here. In fact, our main result and our main application can be obtained as corollaries of the main result of \cite{spadaro}, although the boundary value problem is not explicitly mentioned in \cite{spadaro}.\\
\\
We believe it to be of independent interest to provide a detailed exposition of how a convex integration scheme \`a la Nash and Kuiper can be applied to a map that is not strictly short (this needs some care in coupling the iteration scheme with an appropriate sequence of cut-off functions) and more importantly how such an iteration scheme can be applied even if the relevant set (one has to usually choose a family of loops in a certain path-connected set) is disconnected.\\
\\
The Nash-Kuiper theorem (see \cite{nash2,kuiper}) uses an iteration scheme that starts with a \emph{short map} $u:(M^n,g)\to \R^q$. Recall that $u$ is called \emph{short}, provided $g-u^*g_0$ is a positive definite tensor field on $M$. The iteration scheme of Nash and Kuiper then adds \emph{corrugations} ($q=n+1$) or \emph{spirals} ($q\geqslant n+2$) at all scales to correct the metric error successively while controlling the $C^1$-norm of the perturbed maps during the process. This leads to the convergence in $C^1$. The basic building block is a family of loops with average zero in a suitable $(q-n)$-dimensional sphere. The method does not apply to the equidimensional case, since here, the sphere is degenerate and consists of two isolated points. However, the points define an interval $\I\subset\R$ and one may choose a loop in $\I$ that is concentrated on the two boundary points (see Lemmata \ref{regularcorrugation} and \ref{singularcorrugation}  for precise statements). Using this loop we obtain a corrugation function that leads to weaker estimates than the corresponding ones in codimension one, but still ensure the required Lipschitz regularity of a solution to \eqref{dissproblemrelaxed}.\\
\\
In local coordinates, the setting for problem \eqref{dissproblemrelaxed} can be reformulated as follows: Consider an $n$-polytope $(P,g)$ in $\R^n$ with an appropriate metric $g$ such that the origin is contained in $\mathring P$ and let the isometric immersion $f:B\to\R^{n}$ be prescribed on $B\coloneqq P\cap (\R^{n-1}\times \{0\})$.

\begin{definition}The sets $P\cap (\R^{n-1}\times\R_{\geqslant 0})$ and $P\cap (\R^{n-1}\times\R_{\leqslant 0})$ are called \emph{one-sided neighborhoods} of $B$.
\end{definition}

\begin{definition}\label{C_p-space}
Let $\bar\Omega$ be a one-sided neighborhood of $B$. A map $u$ belongs to $C_p^\infty(\bar\Omega,\R^n)$ if there is a finite simplicial decomposition of $\bar\Omega$ into non-degenerate $n$-simplices such that the restriction of $u$ onto each of the simplices is smooth. Similarly, a map $u$ belongs to the space $\mathrm{Aff}_p(\bar\Omega,\R^n)$ if there is a finite simplicial decomposition of $\bar\Omega$ into non-degenerate $n$-simplices such that the restriction of $u$ onto each of the simplices is affine.\end{definition}

\begin{definition}\label{subsolution}
Let $\bar\Omega$ be a one-sided neighborhood of $B$. A map $$u\in C_p^\infty(\bar\Omega,\R^n)\cap C^0(\bar\Omega,\R^n)$$ is called \emph{short map adapted to $(f,g)$}, if  $u|_B=f$ and $g-u^*g_0\geqslant 0$ in the sense of quadratic forms with equality on $B$ only (i.e.\ $g-u^*g_0$ is positive definite on $\bar\Omega\setminus B$ and zero on $B$).
\end{definition}

We are now ready to state our main result:

\begin{theorem}\label{maintheorem}
Let $u:\bar\Omega\to\R^n$ be a short map adapted to $(f,g)$. Then for every $\varepsilon>0$, there exists a Lipschitz map $v:\bar\Omega\to\R^n$ whose singular set has Hausdorff dimension $n-1$ satisfying $v|_B=f$, $v^*g_0=g$ $\mathcal L^n$-a.e.\ and $\|u-v\|_{C^0(\bar\Omega)}< \varepsilon$.
\end{theorem}
We obtain a Corollary regarding isometric extensions of the standard inclusion $$\iota: S^1\hookrightarrow \bar D^2\subset\R^2$$ to maps $S^2\to \bar D^2$, where $\bar D^2$ denotes the closed two-dimensional unit disk and $S^1$ is the equator of $S^2$. Here $\mu_{S^2}$ denotes the standard measure on $S^2$.
\begin{coro}[Isometric Collapse of $S^2$]\label{collapse}
There exist infinitely many Lipschitz maps $v:S^2\to\bar D^2$ satisfying $v|_{S^1}=\iota$ and $v^*g_0=g_{S^2}$ $\mu_{S^2}$-a.e.
\end{coro}

This Corollary shows the strong interaction between codimension and regularity: The standard inclusion $S^1\hookrightarrow\R^2\times\{0\}\hookrightarrow\R^3$ can be extended to an isometric immersion $v\in C^{1,\alpha}(S^2,\R^3)$ in a unique way (up to reflection across the $\R^2\times\{0\}$-plane) provided $\alpha>\frac{2}{3}$ (see \cite{borisov1,borisov2,borisov4,borisov3,borisov5,delellis}), but infinitely many isometric extensions $v\in C^1(S^2,\R^3)$ exist (see \cite{wasem}). The recent work of Cao and Sz\'ekelyhidi \cite{caolaszlo} indicates how this flexibility result can be improved to $v\in C^{1,\alpha}(S^2,\R^3)$, provided $\alpha<\frac17$. It is conjectured in \cite{caolaszlo} that for $n=2$, the techniques developed in \cite{inauen} can be adapted to improve the H\"older exponent further to $\alpha<\frac15$. In the case of codimension $0$, we will show that no isometric extension $v\in C^1(S^2,\R^2)$ exists. More precisely, we will show that the maps $v$ from Corollary \ref{collapse} cannot be locally $C^1$ nor locally injective, hence the restriction to Lipschitz maps is not excessive.

\subsection*{Acknowledgements}
I would like to sincerely thank Norbert Hungerb\"uhler for many helpful discussions about this work and Thomas Mettler for repeatedly taking me under his wings -- geile Schnitzer!

\section{Obstructions}
We will first show that curvature is an obstruction against merely differentiable (not necessarily $C^1$) isometric immersions in the equidimensional case which leads to the choice of the Lipschitz regularity in \eqref{dissproblemrelaxed}.
\begin{theorem}\label{c1curvature}
An $n$-dimensional Riemannian manifold $(M^n,g)$ can be locally isometrically embedded by a differentiable map into $(\R^n,g_0)$ if and only if $g$ is flat and in this case, the map is in fact of class $C^\infty$.\end{theorem}

\begin{proof} The if part of the statement is a classical theorem in differential geometry (see for example \cite[Theorem 3.1]{taylor}). A local isometric immersion $f$ is locally a distance preserving homeomorphism by the differentiable local inversion theorem (see \cite{saintraymond}) and the condition $f^*g_0=g$. Such a map must be a local $C^\infty$-isometric diffeomorphism by \cite[Theorem 2.1]{taylor} since $g$ is smooth by assumption. Hence the statement is a simple consequence of the classical fact that the Riemannian curvature tensor is preserved under $C^\infty$-isometries.\end{proof}

Recall that $\Sigma$ is a hypersurface of an $n$-dimensional Riemannian manifold $(M,g)$ and $f:\Sigma\to\R^{n}$ is an isometric immersion we seek to extend to a neighborhood $U$ of a point in $\Sigma$. We denote by $A\in\Gamma(\mathrm S^2(T^*\Sigma)\otimes N\Sigma)$ the second fundamental form of $\Sigma$ in $M$, and let $h(X,Y)\coloneqq g(\nu,A(X,Y))$, where $\nu\in\Gamma(N\Sigma)$ is the unique (up to sign) unit normal vector field. Let further $\bar A\in\Gamma(\mathrm S^2(T^*\Sigma)\otimes f^*N\bar \Sigma)$ be the second fundamental form of $\bar \Sigma\coloneqq f(\Sigma)$ in $\R^n$ with associated scalar fundamental form $\bar h(\cdot\,,\cdot)\coloneqq\langle \bar A(\cdot\,,\cdot),\bar\nu\rangle$, where $\bar\nu\in\Gamma(f^*N\bar \Sigma)$ is again a unit normal vector that is unique up to sign. The following proposition is a variant of \cite[Proposition 1]{wasem}:
\begin{prop}[Lipschitz-Obstruction]\label{lipschitzobstruction}
If there exists a unit vector $v\in T_p\Sigma$ such that $|h(v,v)|_g>|\bar h(v,v)|$, no isometric Lipschitz extension $u:U\to\R^{n}$ can exist.
\end{prop}
\begin{proof}
We argue by contradiction. Suppose $u$ exists and let (for $\varepsilon>0$ small enough) $\gamma:[0,\varepsilon)\to \Sigma\cap U$ be a geodesic with $\gamma(0)=p$ and $\dot\gamma(0)=v$ such that $d_M(p,\gamma(t))$ can be realized by a minimizing geodesic $\sigma:[0,1]\to U$ for all $t$.

\textit{Claim: }For every $\delta>0$ there exists a curve $c:[0,1]\to U$ joining $p$ and $\gamma(t)$ such that $u$ is differentiable in $c(t)$ $\mathcal L^1$-a.e. in $[0,1]$ and such that
$$
\int_0^1|\dot c(t)|_g\,\mathrm dt<d_M(p,\gamma(t))+\delta.
$$
\textit{Proof of the claim: } Let $D\subset U$ be the set where $u$ is differentiable,
$$Z\coloneqq [0,1]\times B_\rho^{n-1}(0)$$and consider the map
$$z:Z\to U, (s,x)\mapsto \exp_{\sigma(s)}(s(s-1)x),$$
where $x\in B_\rho^{n-1}(0)\subset \dot\sigma(s)^\perp\subset T_{\sigma(s)}U$ (here $B_\rho^{n-1}(0)$ is identified with the geodesic $\rho$-ball centered at the origin in the orthogonal complement of $\dot\sigma(s)\in T_{\sigma(s)}U$).
If $\rho>0$ is small enough, the restriction of $z$ onto $(0,1)\times B_\rho^{n-1}(0)$ is a diffeomorphism onto its image, hence $\widetilde D\coloneqq z^{-1}(D)$ has full $\mathcal L^n$-measure in $Z$. We obtain
$$
\int_Z |\partial_sz|_g\,\mathrm ds\,\mathrm dx = \int_{Z\cap\widetilde D}|\partial_s z|_g\,\mathrm ds\,\mathrm dx = \int_{B^{n-1}_\rho(0)}\int_{0}^{1}\chi_{\widetilde D}|\partial_s z|_g\,\mathrm ds\,\mathrm dx
$$
hence for $\mathcal L^{n-1}\text{-a.e.}$ $x$ it holds that
$$
\int_0^1\chi_{\widetilde D}(s,x)|\partial_s z(s,x)|_g\,\mathrm ds = \int_0^1|\partial_s z(s,x)|_g\,\mathrm ds.
$$
Since
$$
\lim_{x\to 0} \int_0^1|\partial_s z(s,x)|_g\,\mathrm ds =d_M(p,\gamma(t)),
$$
we can choose an appropriate $x$ such that $c(t)\coloneqq z(t,x)$ has the desired properties. This finishes the proof of the claim. Let $\bar p = f(p)$ and $\bar \gamma = f\circ \gamma$. Using the claim we obtain
$$\begin{aligned}
|u(c(1))-u(c(0))|\leqslant \int_0^1\left|\frac{\mathrm d}{\mathrm dt}(u\circ c)(t)\right|\,\mathrm dt = \int_0^1|\dot c(t)|_g\,\mathrm dt<d_M(p,\gamma(t))+\delta.
\end{aligned}$$
We conclude that $d_{\R^{n}}(\bar p,\bar \gamma(t))\leqslant d_M(p,\gamma(t))$, since $\delta$ was arbitrary. The rest of the proof is exactly the same as in \cite[Proposition 1, p.\ 754]{wasem}.
\end{proof}

\section{Adapted Short Maps}
For the construction of adapted short maps, we refer to a variant of  \cite[Proposition 2, p.\ 755]{wasem}:
\begin{prop}\label{constructionofsubsolutions}
Let $f:\Sigma\to\R^{n}$ be an isometric immersion. Suppose there exist unit normal fields $\nu\in\Gamma(N\Sigma)$ and $\bar\nu\in\Gamma(f^*N\bar \Sigma)$ such that $h(\cdot\,,\cdot)-\bar h(\cdot\,,\cdot)$ is positive definite, then around every $p\in \Sigma$, there exists a short map adapted to $(f,g)$.
\end{prop}

\subsection{Approximation of Short Maps by Piecewise Affine Maps}

We need to introduce the notion of an \emph{adapted piecewise affine short map} and therefore the following approximation result (see \cite{saigal} for a proof):
\begin{prop}\label{approximationprop}
Let $u\in C_p^\infty (\bar\Omega,\R^n)\cap C^0(\bar\Omega,\R^n)$. For every $\varepsilon>0$ there exists a map $v\in\mathrm{Aff}_p(\bar\Omega,\R^n)\cap C^0(\bar\Omega,\R^n)$ such that $\|u-v\|_{C^0(\bar\Omega)}+\|\nabla u-\nabla v\|_{L^\infty(\bar\Omega)}<\varepsilon.$
\end{prop}
\begin{remark}\label{adaptedapprox}
Note that if an adapted short map $u$ is approximated by a piecewise affine map $v$, we cannot ensure that $u|_B=v|_B$ since this would require $u|_B$ to be piecewise affine already. To circumvent this problem, let $\bar\Omega_\ell\coloneqq\{x\in\bar\Omega, \operatorname{dist}(x,B)\leqslant\ell\}$ and define $\eta_\ell\in C^\infty(\bar\Omega,[0,1] )$ to be
$$
\eta_\ell(x)\coloneqq\begin{cases}\hfill 0,& \text{if }x\in \bar\Omega_{\ell/2}\\ \hfill 1, & \text{if }x\in \overline{\Omega_{\ell}^c}\coloneqq\left\{x\in\bar\Omega,\operatorname{dist}(x,B)\geqslant{\ell}\right\}\end{cases}
$$
and $0<\eta_\ell < 1$ elsewhere.
\begin{figure}[H]
\begin{center}
\psfragfig[scale=0.75]{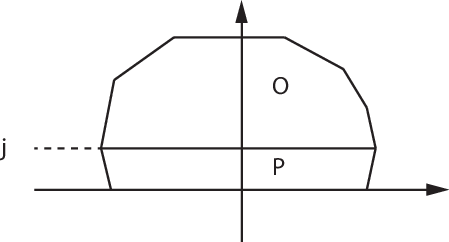}
\end{center}
\end{figure}
\begin{definition}\label{piecewiseadapted}
For fixed $\ell>0$, decompose $\bar\Omega_{\ell/2}$ and the closure of its complement in $\bar\Omega$ separately into non-degenerate simplices, approximate $u$ by $v$ in the sense of Proposition \ref{approximationprop} and finally replace $v$  by $w\coloneqq u+\eta_{\ell/2}(v-u)$. This map has the property that $w\equiv u$ on $\bar\Omega_{\ell/4}$ and the restriction of $w$ to $\overline{\Omega_{\ell/2}^c}$ is a piecewise affine map. If such a map has in addition the property that $g-\nabla w^T\nabla w> 0$ $\mathcal L^n\text{-a.e.}$, but $\nabla w^T\nabla w = g$ on $B$, we will call it a \emph{piecewise affine short map adapted to $(f,g)$}.
\end{definition}
The estimates
$$\begin{aligned}\|w-u\|_{C^0(\bar\Omega)}&\leqslant \|u-v\|_{C^0(\bar\Omega)}\\ \|\nabla w - \nabla u\|_{L^\infty(\bar\Omega)} & \leqslant \|\nabla \eta_{\ell/2}\|_{C^0(\bar\Omega)}\|u-v\|_{C^0(\bar\Omega)}+\|\nabla u-\nabla v\|_{L^\infty(\bar\Omega)}\end{aligned}$$
ensure that we can approximate an adapted short map and its first derivatives by an adapted piecewise affine short map.\end{remark}

\section{Convex Integration}
We start with a piecewise affine short map $u:\bar\Omega\to\R^{n}$ adapted to $(f,g)$, that is piecewise affine on $\overline{\Omega_{\ell/2}^c}$ and decompose the metric defect into a sum of primitive metrics (see \cite[p.\ 202, Lemma 1]{laszlo} for an explanation) as
$$
(g-u^*g_0)_x=\sum_{k=1}^ma_k^2(x)\nu_k\otimes\nu_k,
$$
where the $a_k^2$ are nonnegative on $\bar\Omega\setminus B$, zero on $B$, belong to $C^\infty_p(\bar\Omega,\R^n)$ and extend continuously to $B$. The sum is locally finite with at most $m_0\leqslant m$ terms being nonzero for a fixed point $x$. The $\nu_k\in S^{n-1}$ are fixed unit vectors. We intend to correct this metric defect by successively adding \emph{primitive metrics}, i.e.\ metric terms of the form $a^2\nu\otimes\nu$. Adding a primitive metric is done in a \emph{step}. A \emph{stage} consists then of $m$ steps, where the number $m\in\N$ is finite but may change from stage to stage. Fix orthonormal coordinates in the target so that the metric $u^*g_0$ can be written as $\nabla u^T\nabla u$, where $\nabla u = (\partial_ju^i)_{ij}$. For a specific unit vector $\nu\in S^{n-1}$ and a nonnegative function $a\in C^\infty(\bar\Omega)$ we aim at finding $v:\bar\Omega\to \R^{q}$ satisfying $\nabla v^T\nabla v \approx \nabla u^T\nabla u + a^2\nu\otimes \nu$. Nash solved this problem using the \emph{Nash twist}, i.e.\ an ansatz of the form 
\begin{equation}\label{nashtwist}
v(x) = u(x) +\frac{1}{\lambda}\bigg[\mathrm N_1(a(x),\lambda \langle x,\nu\rangle)\beta_1(x)+\mathrm N_2 (a(x),\lambda \langle x,\nu\rangle)\beta_2(x)\bigg],
\end{equation}
where $\mathrm N(s,t)\coloneqq s (-\sin t,\cos t)$ satisfies the circle equation $\partial_t\mathrm N_1^2+\partial_t\mathrm N_2^2=s^2$ and $\beta_i$ are mutually orthogonal unit normal fields, requiring thus codimension at least two ($q\geqslant n+2$). The improvement to codimension one was first achieved by Kuiper \cite{kuiper} with the use of a different ansatz (\emph{strain}). We will present the \emph{corrugation} introduced by Conti, de Lellis and Sz\'ekelyhidi \cite{delellis} since it will be illustrative for the codimension zero case. Define
$$\widetilde \xi\coloneqq\nabla u \cdot \left(\nabla u^T\nabla u\right)^{-1}\cdot \nu\text{ and }\widetilde \zeta\coloneqq\star\left(\partial_1u\wedge\partial_2u\wedge\ldots\wedge\partial_nu\right),$$
where $\star$ denotes the Hodge star with respect to the usual metric and orientation in $\R^{n
+1}$ and let
\begin{equation*}\xi\coloneqq\frac{\widetilde\xi}{|\widetilde\xi|^2},\qquad\zeta\coloneqq\frac{\widetilde\zeta}{|\widetilde\zeta||\widetilde \xi|}.\end{equation*}
The corrugation then has the form
\begin{equation}\label{ansatzdelellis}
v(x) = u(x)+\frac{1}{\lambda}\bigg[\Gamma_1(a(x)|\widetilde \xi(x)|,\lambda \langle x,\nu\rangle)\xi(x) + \Gamma_2(a(x)|\widetilde \xi(x)|,\lambda \langle x,\nu\rangle)\zeta(x)\bigg],
\end{equation}
where $\Gamma\in C^\infty(\R \times S^1,\R^2), (s,t)\mapsto \Gamma(s,t)$ is a family of loops satisfying the circle equation
$\left(\partial_t\Gamma_1+1\right)^2+\partial_t\Gamma_2^2 = 1+s^2$. Observe that both, $\mathrm N$ and $\Gamma$ satisfy the average condition
$$\oint_{S^1}\partial_t\mathrm N\,\mathrm dt =0\text{ and }\oint_{S^1}\partial_t\Gamma\,\mathrm dt =0$$
and -- as a crucial ingredient to control the $C^1$-norm during the iteration -- the following $C^1$-estimates (see \cite[Lemma 4, p.\ 762]{wasem} for the proof of the second estimate):
\begin{equation}\label{c1control}\begin{aligned}|\partial_t\mathrm N(s,t)|&=|s|\\
|\partial_t \Gamma(s,t)|& \leqslant\sqrt{2}|s|\end{aligned}\end{equation}
for all $(t,s)\in \R\times S^1$. In the equidimensional case, where there is no normal vector at all, we follow the strategy of Sz\'ekelyhidi (see \cite{laszlo}) and let $\widetilde \xi\coloneqq \nabla u^{-T}\cdot\nu$ and $\xi\coloneqq \widetilde \xi |\widetilde \xi|^{-2}$. A similar ansatz like \eqref{nashtwist} or \eqref{ansatzdelellis} is
\begin{equation}\label{lipschitzansatz}
v(x)=u(x)+\frac{1}{\lambda}\mathrm L(a(x)|\widetilde \xi(x)|,\lambda \langle x,\nu\rangle)\xi,
\end{equation}
where $\mathrm L:\R\times S^1\to\R^2, (s,t)\mapsto \mathrm L(s,t)$ is a smooth family of loops still to be constructed. Direct computations show that $\nabla v = \nabla u +\partial_t \mathrm L \xi\otimes \nu + O(\lambda^{-1})$ and therefore by the definition of $\xi$:
$$
\nabla v^T\nabla v = \nabla u^T\nabla u + \frac{1}{|\widetilde \xi|^2}(2\partial_t\mathrm L+\partial_t\mathrm L^2)\nu\otimes\nu + O(\lambda^{-1}).
$$
In order to obtain $\nabla v^T\nabla v = \nabla u^T\nabla u + a^2\nu\otimes\nu + O(\lambda^{-1})$, $\partial_t\mathrm L$ needs to satisfy the circle equation
$$(1+\partial_t\mathrm L)^2=1+|\widetilde \xi|^2a^2\eqqcolon1+s^2.$$
Since the circle here is zero-dimensional, it consists of two isolated points and there is no smooth map $\partial_t\mathrm L$ to that circle being zero in average. We will circumvent this problem by replacing the equality in the circle equation by a pointwise and an average inequality (see Lemma \ref{regularcorrugation} and Figure \ref{convexintegration}). Note that the definition of $\xi$ in \eqref{lipschitzansatz} requires $u$ to be immersive. If $\nabla u$ does not have full rank, we will choose $\xi$ to be a unit vector field in $\ker \left(\nabla u^T\right)$ which will lead to the circle equation $\partial_t\mathrm L^2=s^2$ (see Lemma \ref{singularcorrugation}).

\begin{lemma}[Regular Corrugation]\label{regularcorrugation}
For every $\varepsilon>0$ and $c>0$ there exists a map $$\mathrm L\in C^\infty([0,c]\times S^1),(s,t)\mapsto \mathrm L(s,t)$$ satisfying the following conditions:
\begin{align}
(1+\partial_t\mathrm L)^2 & \leqslant 1+s^2\label{pointwiseuplipschitz}\\
\frac{1}{2\pi}\oint_{S^1}\left(s^2-\partial_t\mathrm L^2\right)\,\mathrm dt & <\varepsilon\label{averagelowlipschitz}\\
\frac{1}{2\pi}\oint_{S^1}\partial_t\mathrm L\,\mathrm dt & = 0
\nonumber
\end{align}
\end{lemma}

\begin{proof}
Consider the $2\pi$-periodic extension of the function $\hat p_s:[0,2\pi]\to[-1,1]$, where $s\in[0,c]$:
$$\hat p_s(t)\coloneqq\begin{cases}\hfill-1, & x\in\left[\frac{\pi}{2}\left(1+\frac{1}{\sqrt{1+s^2}}\right), \frac{\pi}{2}\left(3-\frac{1}{\sqrt{1+s^2}}\right)\right]\\ \hfill 1 &\text{else.}\end{cases}$$
For $0<\varepsilon<\frac{1}{2}$ fixed, choose $0<\delta<\frac{\varepsilon\pi}{2(1+c^2)}$ and let $\varphi_\delta$ denote the usual symmetric standard mollifier. The function $$p:\R\times S^1\to\R, (s,t)\mapsto \varphi_{\delta} * \hat p_s(t)$$ is smooth and a direct computation using Fubini's theorem implies
$$
\frac{1}{2\pi}\oint_{S^1}p(s,t)\,\mathrm dt = \frac{1}{\sqrt{1+s^2}}.
$$
Since $p^2(s,\cdot)$ equals one on a domain of measure at least $2\pi-4\delta$ on each period we get
$$
\frac{1}{2\pi}\oint_{S^1}p^2(s,t)\,\mathrm dt> 1 - \frac{2\delta}{\pi} > 1-\frac{\varepsilon}{1+s^2}.
$$
The function $\mathrm L: \R\times S^1\to \R$
$$
\mathrm L(s,t)\coloneqq\int_0^t \left(\sqrt{1+s^2}\cdot p(s,u)-1\right)\,\mathrm du
$$
then has all the desired properties.
\end{proof}

\begin{lemma}[Singular Corrugation]\label{singularcorrugation}
For every $\varepsilon>0$ and $c>0$ there exists a map $$\widetilde {\mathrm L}\in C^\infty([0,c]\times S^1), (s,t)\mapsto \widetilde {\mathrm L}(s,t)$$ satisfying the following conditions:
\begin{align}
\partial_t \widetilde{\mathrm L}^2 & \leqslant s^2\label{pointwiseuplipschitzsing}\\
\frac{1}{2\pi}\oint_{S^1}\left(s^2-\partial_t \widetilde {\mathrm L}^2\right)\,\mathrm dt & <\varepsilon\label{averagelowlipschitzsing}\\
\frac{1}{2\pi}\oint_{S^1}\partial_t\widetilde {\mathrm L}\,\mathrm dt & = 0\nonumber
\end{align}
\end{lemma}

\begin{proof}
Consider the convolution of the $2\pi$-periodic extension of the map
$$[0,2\pi]\ni t\mapsto \begin{cases}\hfill-s, & t\in [\frac{\pi}{2},\frac{3\pi}{2}] \\ \hfill s&\text{else} \end{cases}$$
with $\varphi_\delta$, where $\delta<\frac{\varepsilon\pi}{2c^2}$. This convolution gives rise to a map $p(s,t)$. The map
$$\widetilde {\mathrm L}(s,t)\coloneqq\int_0^t p(s,u)\,\mathrm du$$
then has all the desired properties.
\end{proof}
\begin{center}
\begin{figure}[H]\label{convexintegration}
\psfragfig[scale=0.75]{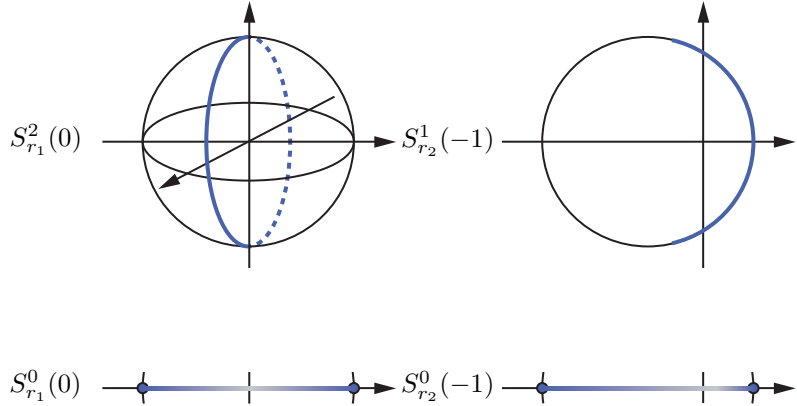}
\caption{From left to right and from top to bottom the picture illustrates the corrugation map used by Nash, the one used by Conti, de Lellis and Sz\'ekelyhidi, and the maps $\widetilde{\mathrm L}$ and $\mathrm L$ from the Lipschitz case. The radii of the spheres are given by $r_1=|s|$ and $r_2=\sqrt{1+s^2}$.}
\end{figure}
\end{center}

Conditions \eqref{averagelowlipschitz} and \eqref{averagelowlipschitzsing} oppose a $C^1$-estimate for $\mathrm L$ and $\widetilde{\mathrm L}$ analogous to \eqref{c1control}, hence the $C^1$-norms of the maps cannot be controlled during the iteration, but we will use a suitable $L^2$-estimate due to Sz\'ekelyhidi using an identity relating the first derivatives to the trace of the metric defect and integration by parts later on (see Proposition \ref{laszloestimate}).

\section{Iteration} Let $\ell>0$ and let $u$ be a piecewise affine short map adapted to $(f,g)$ in the sense of Definition \ref{piecewiseadapted} and decompose the metric defect as
$$
g-u^*g_0 = \sum_{k=1}^m a_k^2\nu_k\otimes\nu_k.
$$
Here the $a_k$ are nonnegative functions that are piecewise constant on $\overline{\Omega_{\ell/2}^c}$. Since $u$ is already isometric on $B$, we will add a ``cut-off'' error $\eta_{\ell}^2(g-u^*g_0)$ using the building blocks $\mathrm L$ and $\widetilde{\mathrm L}$.
\subsection{$k$-th Step} Let $u_{k-1}$ be a piecewise affine short map adapted to $(f,g)$ such that $u_{k-1}$ is piecewise affine on $\overline{\Omega_{\ell/2}^c}$. We introduce a map $\Theta_k\in C^\infty(\bar\Omega,[0,1])$ in the $k$-th step that is associated to $u_{k-1}$ as follows: Let $\{ S_i\}_i$ be the simplicial decomposition of $\bar\Omega$ according to $u_{k-1}$ and let $U_k$ be an open neighborhood of
$$
K_k\coloneqq\bar\Omega \setminus \bigcup_{i}\mathring S_i
$$
and set $\Theta_k = 1$ on $U_k^c$ and $\Theta_k=0$ on $K_k$. Observe that $U_k$ can be chosen to have arbitrary small Lebesgue measure. Let $0<\delta<1$ (the exact value will be determined later). We will now discuss the step for a simplex $S\subset\overline{\Omega_{\ell/2}^c}$. The restriction of $u_{k-1}$ to $S$ is an affine function. If $\nabla u_{k-1}$ is regular on $S$, let
$$\begin{aligned}\widetilde \xi_k&\coloneqq\nabla u_{k-1} (\nabla u_{k-1}^T\nabla u_{k-1})^{-1}\cdot\nu_k\\ \xi_k & \coloneqq\widetilde \xi_k|\widetilde \xi_k|^{-2}\\ s_k&\coloneqq(1-\delta)^{\sfrac{1}{2}}\Theta_k\eta_{\ell}a_k|\widetilde\xi_k|\\ \mathrm L_k(x)&\coloneqq\mathrm L\left(s_k,\lambda_k\langle x,\nu_k\rangle\right).\end{aligned}$$
and define $$u_k\coloneqq u_{k-1}+\frac{1}{\lambda_k}\mathrm L_k\xi_k.$$We have $u_k\in C^\infty(S,\R^{n})$ and $\|u_{k-1}-u_k\|_{C^0(S)}$ can be made arbitrarily small provided the free parameter $\lambda_k$ is large enough. For the Euclidean metric pulled back by $u_k$, we find
$$
\nabla u_{k}^T\nabla u_k = \nabla u_{k-1}^T\nabla u_{k-1} + \frac{1}{|\widetilde\xi_k|^2}\left(2\partial_t\mathrm L_k+\partial_t \mathrm L_k^2\right)\nu_k\otimes\nu_k+O\left(\lambda_k^{-1}\right).
$$
If $\nabla u_{k-1}$ is singular, choose $\xi_k\in \ker(\nabla u_{k-1}^T)$ to be a unit vector, let
$$\begin{aligned}s_k&\coloneqq(1-\delta)^{\sfrac{1}{2}}\Theta_k\eta_{\ell}a_k\\ \widetilde{\mathrm L}_k(x)&\coloneqq\widetilde{\mathrm L}\left(s_k,\lambda_k\langle x,\nu_k\rangle\right)\end{aligned}$$ and define $$u_k\coloneqq u_{k-1}+\frac{1}{\lambda_k}\widetilde{\mathrm L}_k\xi_k.$$As in the regular case, $u_k\in C^\infty(S,\R^{n})$ and $\|u_{k-1}-u_k\|_{C^0(S)}$ can be made arbitrarily small. For the Euclidean metric pulled back by $u_k$ we find
$$
\nabla u_{k}^T\nabla u_k = \nabla u_{k-1}^T\nabla u_{k-1} + \partial_t\widetilde{\mathrm L}_k^2\nu_k\otimes\nu_k+O\left(\lambda_k^{-1}\right).
$$
\subsection{Stage} We approximate the resulting map after each step by an adapted piecewise affine short map. This introduces a further subdivision of the simplicial decomposition of $\bar\Omega$ after each step. In each step we will leave the map from the foregoing step unchanged near $B$ (due to the cut-off by $\eta_\ell$) and near the $(n-1)$-skeleton of the simplicial decomposition (thanks to $\Theta_k$). This procedure does not allow for a pointwise control of the new metric error, but we are still able to bound it in an integral sense.
\begin{prop}[Stage]\label{lipschitzstage}
Let $u$ be a short map adapted to $(f,g)$. Then for any $\varepsilon>0$, there exists a piecewise affine short map $\widetilde u$ adapted to $(f,g)$ satisfying
\begin{align}
\|u-\widetilde u\|_{C^0(\bar\Omega)}&<\varepsilon\label{c0lipschitz}\\
\int_{\bar\Omega}\operatorname{tr}\left(g-\nabla\widetilde u^T\nabla\widetilde u\right)\,\mathrm dx & < \varepsilon\label{metriclipschitz}
\end{align}
\end{prop}

\begin{proof}
Since $\int_{\bar\Omega_\ell}\operatorname{tr}(g-\nabla  u^T\nabla  u)\,\mathrm dx$ converges to zero, as $\ell\to 0$, there exists $\ell>0$ small enough such that after approximating $u$ by a piecewise affine short map $\widetilde u_0$ adapted to $(f,g)$ that is piecewise affine on $\overline{\Omega_{\ell/2}^c}$ we get
\begin{equation}\label{choosingelllipschitz}
\int_{\bar\Omega_\ell}\operatorname{tr}(g-\nabla \widetilde u_0^T\nabla \widetilde u_0)\,\mathrm dx<\frac{\varepsilon}{7}.
\end{equation}
Choose (for $\ell$ now fixed) $\delta$ such that
\begin{align}
\delta \operatorname{id} &< (g-\widetilde u_0^*g_0)|_{\overline{\Omega_{\ell/2}^c}}\quad\text{$\mathcal L^n$-a.e.}\label{choosingdeltalipschitz}\\
\delta &< \frac{\varepsilon}{7}\left(\int_{\overline{\Omega_{\ell}^c}}\operatorname{tr}(g-\nabla \widetilde u_0^T\nabla \widetilde u_0)\,\mathrm dx\right)^{-1}\label{intdelta}.
\end{align}
We use the step iteratively to produce a sequence of maps starting with $\widetilde u_0$. After a step, say the $k$-th, we approximate the resulting map $u_k$ by an adapted piecewise affine short map $\widetilde u_k$ that is piecewise affine on $\overline{\Omega_{\ell/2}^c}$ and leave it unchanged on $\bar\Omega_{\ell/2}$.\\
After $m$ steps, set $\widetilde u\coloneqq\widetilde u_m$. Choosing the free parameter $\lambda_k$ sufficiently large in each step together with suitable approximations by piecewise affine short maps proves \eqref{c0lipschitz}.\\

On each simplex (in $\overline{\Omega_{\ell/2}^c}$) of the simplicial decomposition corresponding to the map $\widetilde u$ we performed $m$ steps and since after each step, the resulting map was affine on that simplex, we used (depending on whether $\nabla \widetilde u_{k-1}$ was regular or singular) the ``regular'' or the ``singular'' $k$-th step. This splits the set $\{1,\ldots,m\}$ into $\mathrm{R}$ and $\mathrm{S}$ corresponding to indices $k$ where $\nabla \widetilde u_{k-1}$ was regular and singular respectively. A direct computation on a fixed simplex shows
\begin{equation}\label{metricerrorcomputation}
\begin{aligned}g-\nabla \widetilde u_m^T\nabla \widetilde u_m & = g - \nabla \widetilde u_0^T\nabla \widetilde u_0 + \nabla \widetilde u_0^T\nabla\widetilde  u_0 - \nabla \widetilde u_m^T\nabla \widetilde u_m\\
& = \sum_{k=1}^m \left(a_k^2\nu_k\otimes\nu_k + \nabla \widetilde u_{k-1}^T\nabla\widetilde  u_{k-1} - \nabla  u_k^T\nabla  u_k\right)+A\\
& = \sum_{k\in\mathrm{R}} \left(\left(a_k^2 -|\widetilde \xi_k|^{-2}\left(2\partial_t\mathrm L_k+\partial_t\mathrm L_k^2\right)\right)\nu_k\otimes\nu_k+O(\lambda_k^{-1})\right)\\
& + \sum_{k\in\mathrm{S}} \left(\left(a_k^2 -\partial_t\widetilde{\mathrm L}_k^2\right)\nu_k\otimes\nu_k+O(\lambda_k^{-1})\right)+A,
\end{aligned}\end{equation}
where
$$A\coloneqq\sum_{k=1}^m\left(\nabla  u_k^T\nabla  u_k-\nabla  \widetilde u_k^T\nabla  \widetilde u_k\right)$$
satisfies $\|A\|_{C^0(\bar\Omega)}<\hat\varepsilon$ and $\hat\varepsilon>0$ will be fixed later (this is possible for every $\hat\varepsilon$ by the use of suitable approximations). In order to prove that $\widetilde u$ is a piecewise affine short map adapted to $(f,g)$, first observe that every $u_k$ is piecewise smooth and continuous. This follows from the infinite differentiability in the interior of every simplex and the fact that $u_k$ agrees with $u_{k-1}$ on $K_k$. Now we prove shortness on $\overline{\Omega_{\ell/2}^c}$: We use the computation \eqref{metricerrorcomputation} and the pointwise estimates \eqref{pointwiseuplipschitz}, \eqref{pointwiseuplipschitzsing} and \eqref{choosingdeltalipschitz} to obtain $\mathcal L^n$-a.e.
\begin{equation*}
\begin{aligned}g-\nabla \widetilde u_m^T\nabla \widetilde u_m & = \sum_{k\in\mathrm{R}} \left(\left(a_k^2 -|\widetilde \xi_k|^{-2}\left(2\partial_t\mathrm L_k+\partial_t\mathrm L_k^2\right)\right)\nu_k\otimes\nu_k+O(\lambda_k^{-1})\right) \\ &+ \sum_{k\in\mathrm{S}} \left(\left(a_k^2 -\partial_t\widetilde{\mathrm L}_k^2\right)\nu_k\otimes\nu_k+O(\lambda_k^{-1})\right)+A
\\ & \geqslant  \sum_{k=1}^m \left(1-(1-\delta)\Theta^2_k\eta^2_\ell\right)a_k^2\nu_k\otimes\nu_k+O(\lambda_k^{-1})+A\\
&\geqslant \delta(g-\widetilde u_0^*g_0)+\sum_{k=1}^m O(\lambda_k^{-1})-\hat\varepsilon\operatorname{id}\\
&>\delta^2\operatorname{id}+\sum_{k=1}^m O(\lambda_k^{-1})-\hat\varepsilon\operatorname{id}>0
\end{aligned}\end{equation*}
provided $\hat\varepsilon$ is small enough and the frequencies $\lambda_k$ are large enough. Note that the pullback of the Euclidean metric is not defined on $K_m$.
In order to prove \eqref{metriclipschitz}, we use the following estimates on a simplex $S$ in the regular and singular case respectively:
\begin{align}\label{averagelowconcrete}
\int_{S}\left(s^2_k-|\widetilde \xi_k|^{-2}(2\partial_t\mathrm L_k+\partial_t\mathrm L_k^2)\right)\,\mathrm dx & < \frac{\varepsilon}{7m\operatorname{vol}\overline{\Omega_{\ell/2}^c}}\\
\int_{S}\left(s^2_k-\partial_t \widetilde{\mathrm L}_k^2\right)\,\mathrm dx & < \frac{\varepsilon}{7m\operatorname{vol}\overline{\Omega_{\ell/2}^c}}\label{averagelowconcretesing}
\end{align}
These estimates are direct consequences of \eqref{averagelowlipschitz} and \eqref{averagelowlipschitzsing} and the fact that for every $f\in C^0(\bar\Omega\times S^1)$
$$
\int_{\bar\Omega}f(x,\lambda \langle x,\nu\rangle)\,\mathrm dx \stackrel{\lambda \to \infty}{\longrightarrow}\int_{\bar\Omega}\frac{1}{2\pi}\oint_{S^1}f(x,t)\,\mathrm dt\,\mathrm dx.
$$
This is the content of Proposition \ref{averageintegral} and will be proved in the appendix. Let
$$\begin{aligned}
\int_{\bar \Omega}\operatorname{tr}\left(g-\nabla \widetilde u_m^T\nabla \widetilde u_m\right)\,\mathrm dx  =& \underbrace{\int_{\bar\Omega_{\ell/2}}\operatorname{tr}\left(g-\nabla \widetilde u_m^T\nabla \widetilde u_m\right)\,\mathrm dx}_{\eqqcolon \I_1}+\\
&+\underbrace{\sum_{i=1}^N\int_{S_i}\operatorname{tr}\left(g-\nabla \widetilde u_m^T\nabla \widetilde u_m\right)\,\mathrm dx,}_{\eqqcolon \I_2}\end{aligned}$$
where $N$ is the total number of simplices in the simplicial decomposition of $\overline{\Omega_{\ell/2}^c}$ according to the map $\widetilde u_m$. Since $\widetilde u_m$ and $\widetilde u_0$ agree on $\bar\Omega_{\ell/2}$, we use \eqref{choosingelllipschitz} to obtain $\I_1\leqslant \frac{\varepsilon}{7}$.
We use \eqref{metricerrorcomputation}, \eqref{averagelowconcrete}, \eqref{averagelowconcretesing}, $\hat\varepsilon<\frac{\varepsilon}{7}(\operatorname{vol}\overline{\Omega_{\ell/2}^c})^{-1}$ and a suitable choice of the $\lambda_k$ to obtain ``$\sum O(\lambda_k^{-1})<\frac{\varepsilon}{7}$'' for estimating $\I_2$:
\begin{equation}\label{integralestimate}\nonumber\begin{aligned}\I_2  \leqslant &\sum_{i=1}^N\bigg[\sum_{k\in\mathrm{R}} \int_{S_i}\left(a_k^2 -|\widetilde \xi_k|^{-2}\left(2\partial_t\mathrm L_k+\partial_t\mathrm L_k^2\right)\right)\,\mathrm dx+\\&
+\sum_{k\in\mathrm{S}} \int_{S_i}\left(a_k^2-\partial_t\widetilde{\mathrm L}_k^2\right)\,\mathrm dx\bigg]+\frac{2\varepsilon}{7}\\
 \leqslant &\sum_{k=1}^m \int_{\overline{\Omega_{\ell/2}^c}}(a_k^2-(1-\delta)\Theta^2_k\eta^2_\ell a_k^2)\,\mathrm dx+\frac{3\varepsilon}{7}\eqqcolon\I_3 + \frac{3\varepsilon}{7},
\end{aligned}\end{equation}
Inequality \eqref{metriclipschitz} follows with the use of \eqref{choosingelllipschitz} and \eqref{intdelta}:
$$
\begin{aligned}
\I_3
 \leqslant &\sum_{k=1}^m\int_{\bar\Omega_{\ell}}(a_k^2-(1-\delta)\Theta_k\eta^2_\ell a_k^2)\,\mathrm dx +\sum_{k=1}^m\int_{\overline{\Omega_{\ell}^c}}(a_k^2-(1-\delta)\Theta_k\eta^2_\ell a_k^2)\,\mathrm dx \\
 \leqslant&\sum_{k=1}^m\int_{\bar\Omega_{\ell}}a_k^2\,\mathrm dx +\sum_{k=1}^m\left(\int_{\overline{\Omega_{\ell}^c}\cap U_k^c}\delta a_k^2\,\mathrm dx+\int_{\overline{\Omega_{\ell}^c}\cap U_k}a_k^2\,\mathrm dx\right) \\
 \leqslant &\int_{\bar\Omega_\ell}\operatorname{tr}(g-\nabla \widetilde u_0^T\nabla \widetilde u_0)\,\mathrm dx + \delta \int_{\overline{\Omega_{\ell}^c}}\operatorname{tr}(g-\nabla \widetilde u_0^T\nabla \widetilde u_0)\,\mathrm dx +\\ & + \sum_{k=1}^m\|a_k\|_{L^\infty(\bar\Omega)}\mathcal L^n(U_k)
\end{aligned}
$$
The first two terms of the last line are bounded by $\frac{\varepsilon}{7}$ and a suitable choice of $U_k$ ensures that $\mathcal L^n(U_k)$ is small enough to make the same hold for the third term as well.\end{proof}

The following $L^2$-estimate due to Sz\'ekelyhidi (see \cite{laszlo}) gives a way to control the derivatives during the stage a posteriori: 
\begin{prop}[Sz\'ekelyhidi]\label{laszloestimate}
If $\varepsilon>0$ from the last proposition is small enough, the maps $u$ and $\widetilde u$ satisfy the estimate
\begin{equation}\label{gradientlipschitz}
\|\nabla u-\nabla \widetilde u\|^2_{L^2(\bar\Omega)} \leqslant C\int_{\bar\Omega}\operatorname{tr}\left(g-\nabla u^T\nabla u\right)\,\mathrm dx.
\end{equation}
\end{prop}

\begin{proof}
Use the Definition of the Frobenius inner product to obtain
\begin{equation}\label{frobeniusidentity}
\operatorname{tr}(g-\nabla\widetilde u^T\nabla\widetilde u)=\operatorname{tr}(g-\nabla u^T\nabla u)-2\operatorname{tr}(\nabla u^T(\nabla \widetilde u-\nabla u))-|\nabla\widetilde u-\nabla u|^2.
\end{equation}
Denote by $\upsilon$ the outer unit normal vector of a simplex $S$. We then have the following integration by parts formula (which holds since $u$ is smooth in the interior of $S$):
\begin{equation}\label{integrationbyparts}
\int_S \operatorname{tr}(\nabla u^T(\nabla \widetilde u -\nabla u))\,\mathrm dx = -\int_S \left\langle \Delta u,\widetilde u - u\right\rangle\,\mathrm dx +\int_{\partial S}\operatorname{tr}(\nabla u^T(\widetilde u-u)\otimes \upsilon)\,\mathrm dS
\end{equation}
Formulae \eqref{frobeniusidentity} and \eqref{integrationbyparts} together with the shortness condition $g - \nabla \widetilde u^T\nabla\widetilde u\geqslant 0$ $\mathcal L^n$-a.e.\ imply
$$\begin{aligned}
\|\nabla \widetilde u - \nabla u\|_{L^2(\bar\Omega)}^2\leqslant & \int_{\bar\Omega}\operatorname{tr}(g-\nabla u^T\nabla u)\,\mathrm dx~+ \\ & + C\left(\|\Delta u\|_{L^\infty(\bar\Omega)}+\|\nabla u\|_{L^\infty(\bar\Omega)}\right)\|\widetilde u-u\|_{C^0(\bar\Omega)}.
\end{aligned}$$
Now, if $\varepsilon$ is small enough, we use \eqref{c0lipschitz} to get \eqref{gradientlipschitz}.
\end{proof}
\subsection{Passage to the Limit} We can now use Proposition \ref{lipschitzstage} iteratively to prove our main result, Theorem \ref{maintheorem}:

\begin{proof} Choose a sequence $(\varepsilon_k)_k$ such that $\sum \varepsilon_k\leqslant\varepsilon$, and such that $\sum\sqrt{\varepsilon_k}<\infty$. Since every adapted piecewise affine short map is also an adapted short map, we can apply the previous theorem iteratively starting with $u_0\coloneqq u$. The iteration gives rise to a sequence $u_k$ that is uniformly Lipschitz (shortness implies $|\nabla u_k|\leqslant \sqrt{\operatorname {tr}g}$ $\mathcal L^n$-a.e.) and Cauchy in $C^0(\bar\Omega,\R^n)$ due to \eqref{c0lipschitz}, hence the limit map $v$ is almost everywhere differentiable by Rademacher's theorem and satisfies $\|u_0-v\|_{C^0(\bar\Omega)}<\varepsilon$ by the choice of the $\varepsilon_k$. Because of \eqref{metriclipschitz} we obtain furthermore
$$\lim\limits_{k\to\infty} u_k^*g_0=g \quad\mathcal L^n\text{-a.e.}$$
Choosing $\varepsilon_k$ in each step so small that estimate \eqref{gradientlipschitz} holds, ensures that $\nabla u_k$ is Cauchy in $L^2(\bar\Omega,\R^{n\times n})$ and converges (up to a subsequence) pointwise $\mathcal L^n$-a.e.\ to a limit map $\Lambda$. We have $\nabla v = \Lambda$ $\mathcal L^n$-a.e., since $\Lambda$ is the weak derivative of $v$. Indeed, fix $\phi\in C^\infty_0(\bar\Omega)$. Using the uniform convergence of $u_k$, integration by parts and H\"older's inequality we obtain
$$
\int_{\bar\Omega}v\otimes\operatorname{grad}\phi \,\mathrm dx  = \lim_{k\to\infty}\int_{\bar\Omega}u_k\otimes\operatorname{grad}\phi\,\mathrm dx = -\lim_{k\to\infty}  \int_{\bar\Omega} \phi\nabla u_k \,\mathrm dx = -\int_{\bar\Omega} \phi \Lambda\,\mathrm dx.
$$
We are left to show that $\nabla v^T\nabla v = g$ $\mathcal L^n$-a.e., which follows from
$$
\nabla v^T\nabla v = \Lambda^T\Lambda =\left( \lim_{k\to\infty}\nabla u_k^T\right)\left(\lim_{k\to\infty}\nabla u_k\right)= \lim\limits_{k\to\infty} u_k^*g_0=g \quad\mathcal L^n\text{-a.e.}
$$
Since $u_k|_B=u_{k+1}|_B$ for all $k\in \N$, we find $v|_B=u|_B$. This proves Theorem \ref{maintheorem} up to the statement about the singular set of $v$ (see Remark \ref{remarkc1}).
\end{proof}
\begin{remark}
The density statement from Theorem \ref{maintheorem} can be reformulated as follows: If $X_0$ denotes the space of short maps adapted to $(f,g)$ equipped with the uniform topology, we can consider its $C^0$-closure
$$X=\left\{u:\bar\Omega\to\R^n, u|_B=f, g-\nabla u^T\nabla u \geqslant 0 \text{ $\mathcal L^n$-a.e.}\right\}$$
and consider the functional $\mathcal F:X\to\R$ given by
$$
\mathcal F[u]\coloneqq\int_{\bar\Omega}\operatorname{tr}\left(g-\nabla u^T\nabla u\right)\,\mathrm dx.
$$
Since $X$ consists of uniform Lipschitz functions, $\mathcal F$ is a well-defined and nonnegative upper semicontinuous functional and we obtain as a corollary of Theorem \ref{maintheorem}
\begin{coro}
The zero-set of $\mathcal F$ is dense in $X$.
\end{coro}
\end{remark}

\begin{remark}[Global Results]
The method presented here allows to obtain global results from the local ones by a partition of unity argument. Here, ``global'' means that we want to construct solutions to \eqref{dissproblemrelaxed} on a neighborhood of $\Sigma$ (and not only of a point in $\Sigma$). The global result is obtained exactly as in \cite[Section 7]{wasem}, where the step, the stage and the iteration are replaced by the step, the stage and the iteration of the present article.
\end{remark}

\section{Application}
We are now in position to prove Corollary \ref{collapse}:

\begin{proof}
We first construct a piecewise smooth map from the upper hemisphere $\mathrm H^+$ to the disk $\bar D^2$ that is short everywhere but on the equator, where it is isometric. The following considerations provide such a map. Consider the maps $\Phi\in C^\infty([0,\tfrac{\pi}{2}]\times[0,2\pi],\R^3)$ and $\Psi\in C^\infty([0,\tfrac{\pi}{2}]\times[0,2\pi],\R^3)$ defined by
$$\begin{aligned}
\Phi(\vartheta,\varphi) & \coloneqq \begin{pmatrix}\cos\vartheta\cos\varphi\\ \cos\vartheta\sin\varphi\\ \sin\vartheta\end{pmatrix}\\
\Psi(\vartheta,\varphi) & \coloneqq f(\vartheta)\begin{pmatrix}\cos\vartheta\cos\varphi\\ \cos\vartheta\sin \varphi\end{pmatrix},
\end{aligned}
$$
where $f:[0,\frac{\pi}{2}]\to[\frac{1}{2},1]$ is defined by
$$
\vartheta\mapsto\begin{cases}\frac{\vartheta^2}{2}-\vartheta+1& x\in[0,1]\\ \frac{1}{2}&\text{else}\end{cases}.
$$
Observe that the composition $\Psi\circ\Phi^{-1}$ gives rise to a modified projection from $\mathrm H^+$ to $\bar D^2$ that is short everywhere but on the equator (where it is isometric). Since the hemisphere is diffeomorphic to the disk, we can think of this map as a map $u:\bar D^2\to \bar D^2$ that is short everywhere but on $\partial D^2=S^1\subset \bar D^2$ with respect to a suitable metric in the source domain. Since one can inscribe two polygons $P_1\subset P_2\subset \bar D^2$ such that $\operatorname{dist}(\partial D^2,P_2)=\frac\ell4$ and $\operatorname{dist}(\partial D^2,P_1)=\frac\ell2$ we can approximate $u$ on $P_2$ by a piecewise affine map by Proposition \ref{approximationprop} and get using Remark \ref{adaptedapprox} a map $v$ that is piecewise affine on $P_1$ and corresponds to $u$ on $\partial D^2$. Now we can proceed as in Proposition \ref{lipschitzstage} to get a sequence of maps that converges to a isometric Lipschitz map thanks to Theorem \ref{maintheorem}. Observe that one can ensure that the image of each intermediate map is indeed $\bar D^2$, since during the iteration, the maps are left unchanged near $B$ in every step and by choosing $\lambda \gg 1$ large enough, the claim follows from the $C^0$-closeness of the maps during the iteration. Now the same procedure can be performed for the lower hemisphere $\mathrm H^-$ and the resulting maps can be glued together along the equator $S^1$ in each step, since the iteration does not change the map in a neighborhood of $S^1$. This concludes the proof.
\end{proof}

\begin{remark}\label{remarkc1}
Theorem \ref{c1curvature} shows that the map $v$ constructed in Corollary \ref{collapse} cannot be $C^1$ or even merely differentiable, since the standard metric on $S^2$ has Gaussian curvature $1$. In this perspective, the Lipschitz regularity seems to be optimal for equidimensional isometries. Moreover, it implies that the singular set of the map $v$ is dense in $S^2$ and $v$ cannot even be locally injective. We will now show that the Hausdorff dimension of the singular set equals one: In general, the set where the pullback of $u$ constructed in Theorem \ref{maintheorem} is undefined has Hausdorff dimension $n-1$ or is empty: Observe that the $k$-th stage introduces a finite number $N_k$ of simplices $S_{i,k}$ to the ones that are present from the $(k-1)$-th step and the pullback of $u_k$ is in general undefined on the $(n-1)$-skeleton of the simplicial decomposition given by the $\{S_{i,k}\}_i$. The singular set of $u$ can then be described as
$$
S=\bigcup_{k\in\N}\bigcup_{i=1}^{N_k}\partial S_{i,k}
$$
It follows that $$\dim_\mathcal H \left(S\right) = \sup_{k,i}\dim_\mathcal H( \partial S_{i,k}) = n-1,$$
since the boundary of an $n$-simplex has Hausdorff dimension $n-1$. If all $N_k=0$, then the singular set is of course empty.
\end{remark}
\appendix
\section{Appendix}

\begin{prop}\label{averageintegral}Let $\Omega\subset\R^n$, then for every $f\in C^0(\bar\Omega\times S^1)$ it holds that
$$
\int_{\bar\Omega}f(x,\lambda \langle x,\nu\rangle)\,\mathrm dx \stackrel{\lambda \to \infty}{\longrightarrow}\int_{\bar\Omega}\frac{1}{2\pi}\oint_{S^1}f(x,t)\,\mathrm dt\,\mathrm dx.
$$
\end{prop}

\begin{proof}
The proof follows \cite[pp.\ 13 - 15]{frenod}. Split $[0,2\pi]$ into $m$ intervals of length $\frac{2\pi}{m}$. Let $t_k$ be the center of the $k$-th interval and let $\chi_k$ be its characteristic function extended $2\pi$-periodically onto $\R$. Consider
$$
\widetilde f_m(x,t)\coloneqq\sum_{k=1}^mf(x,t_k)\chi_k(t)
$$
As $m\to\infty$, $\widetilde f_m$ converges uniformly to $f$. For $\nu\in\R^n$ let $\pi_\nu:\R^n\to\R, x\mapsto \langle x,\nu\rangle$ we find as $\lambda\to\infty$
$$\chi_k(\lambda\pi_\nu(~\cdot~))\stackrel{*}{\rightharpoondown}\frac{1}{2\pi}\oint_{S^1}\chi_i(t)\,\mathrm dt = \frac{1}{m}\text{ in }L^\infty(\bar\Omega).$$
This is an immediate consequence of the following weak* convergence (in $L^\infty (S^1)$) for any $f\in C^0(S^1)$ as $\lambda\to\infty$:
$$
f(\lambda~\cdot~)\stackrel{*}{\rightharpoondown}\frac{1}{2\pi}\oint_{S^1}f(t)\,\mathrm dt
$$
(see for example \cite[pp.\ 33 - 37]{cioranescu}). Hence we find
\begin{equation}\label{weakstar}
\int_{\bar\Omega}\widetilde f_m(x,\lambda\pi_\nu(x))\,\mathrm dx\stackrel{\lambda\to\infty}{\longrightarrow}\int_{\bar\Omega}\frac{1}{2\pi}\oint_{S^1}\widetilde f_m(x,t)\,\mathrm dt \,\mathrm dx\end{equation}
Now we write
$$\left|\int_{\bar\Omega}\left(f(x,\lambda \pi_\nu(x)) \,\mathrm dx-\frac{1}{2\pi}\oint_{S^1}f(x,t)\,\mathrm dt\right)\,\mathrm dx\right|\leqslant\I_1+\I_2+\I_3,$$
where
$$\begin{aligned}
\I_1 & = \left|\int_{\bar\Omega}\left(f(x,\lambda \pi_\nu(x))-\widetilde f_m(x,\lambda \pi_\nu(x)) \right)\,\mathrm dx\right|\\
\I_2& = \left|\int_{\bar\Omega}\left(\widetilde f_m(x,\lambda \pi_\nu(x)) -\frac{1}{2\pi}\oint_{S^1}\widetilde f_m(x,t)\,\mathrm dt \right)\,\mathrm dx\right|\\
\I_3& = \left|\int_{\bar\Omega}\left(\frac{1}{2\pi}\oint_{S^1}\left(\widetilde f_m(x,t) - f(x,t)\right)\,\mathrm dt\right)\,\mathrm dx\right|.
\end{aligned}$$
Now fix $\varepsilon>0$. The uniform convergence $\widetilde f_m\to f$ implies the existence of an $m$ such that $\I_1,\I_3\leqslant\frac{\varepsilon}{3}$. Using \eqref{weakstar} we get $\I_2<\frac{\varepsilon}{3}$ by a suitable choice of $\lambda$. This concludes the proof. \end{proof}

\providecommand{\bysame}{\leavevmode\hbox to3em{\hrulefill}\thinspace}
\providecommand{\MR}{\relax\ifhmode\unskip\space\fi MR }
\providecommand{\MRhref}[2]{%
  \href{http://www.ams.org/mathscinet-getitem?mr=#1}{#2}
}
\providecommand{\href}[2]{#2}

\end{document}